\documentclass[11pt,reqno]{amsart}
\usepackage[T1]{fontenc}
\usepackage{amsmath,amssymb,amsthm,mathtools}
\usepackage[margin=1in]{geometry}
\usepackage[colorlinks=true,linkcolor=blue,citecolor=blue,urlcolor=blue]{hyperref}
\numberwithin{equation}{section}
\newtheorem{theorem}{Theorem}[section]

\newtheorem{lemma}[theorem]{Lemma}
\newtheorem{corollary}[theorem]{Corollary}
\theoremstyle{remark}\newtheorem{remark}[theorem]{Remark}
\newcommand{\Hol}{\operatorname{Hol}}

\newcommand{\Id}{\mathrm{I}}
\newcommand{\su}{\mathfrak{su}}
\newcommand{\uu}{\mathfrak{u}}
\newcommand{\Ric}{\operatorname{Ric}}
\newcommand{\tr}{\operatorname{tr}}
\newcommand{\diag}{\operatorname{diag}}
\title[Holonomy under Chern--Ricci flows]{Chern holonomy under the first\\and second Chern--Ricci flows}
\author{Yuan Liao} 
\address{University of California San Diego}
\email{y2liao@ucsd.edu}
\date{ 26 October 2026}

\begin{document}

\begin{abstract}
We compare the Chern holonomy of the first and second Chern--Ricci
flows on compact complex manifolds. For the first flow, we construct
an explicit immortal pluriclosed solution on a complex two-torus with
full and restricted
holonomy \(U(2)\), except at one finite time when both groups are
\(U(1)^2\); its smooth limit has holonomy \(SU(2)\).
The solution is uniformly nondegenerate and has bounds of every
fixed differential order for all time. 
In contrast, the second Chern--Ricci flow preserves full and
restricted Chern holonomy under the natural unitary identification,
including the initial time and nonclosed groups.
The proof combines preservation of a connection reduction,
positive-time analyticity, and a path-transport comparison.
As an application, the optimal minimum mean curvature on a
holomorphic bundle is unchanged when metrics are restricted to
any realizable full and restricted holonomy class.
An explicit projected-curvature source explains the failure
of the corresponding Ricci-flow holonomy argument.
\end{abstract}
\maketitle

\section{Introduction}
Let \(M\) be a complex manifold, and let
\(R_{i\bar j k\bar l}\) denote the curvature of the Chern connection
of a Hermitian metric \(g\).
The two natural Chern--Ricci contractions shall rise distinct
evolution equations, in parallel to Ricci flow.
The first Chern--Ricci flow, usually called the Chern--Ricci flow,
is defined by
\begin{equation}\label{eq:crf-main}
 \partial_t g_{i\bar j}
 =-\Ric^{(1)}_{i\bar j}
 =\partial_i\partial_{\bar j}\log\det g,
 \qquad
 \Ric^{(1)}_{i\bar j}
 =g^{k\bar l}R_{i\bar j k\bar l}.
\end{equation}
The second Chern--Ricci flow is defined by
\begin{equation}
 \partial_t g_{i\bar j}=-\Ric^{(2)}_{i\bar j},
 \qquad
 \Ric^{(2)}_{i\bar j}
 =g^{k\bar l}R_{k\bar l i\bar j}.
\end{equation}

The first Chern--Ricci flow in the non-K\"ahler setting
was first studied by Gill through parabolic complex
Monge--Amp\`ere equations \cite{Gill}.
When the first Bott--Chern class vanishes, his work gives
long-time existence and smooth convergence to a
Chern--Ricci-flat Hermitian metric for every Hermitian
initial metric on a compact complex manifold.
Tosatti and Weinkove subsequently developed the theory
for general compact Hermitian manifolds, characterizing
the maximal existence time in terms of the initial metric
\cite[Theorem 1.2]{TW}.
They also proved that, on a compact manifold with ample
canonical bundle, the flow starting from any Hermitian
metric converges after normalization to the unique
negative K\"ahler--Einstein metric
\cite[Theorem 1.7]{TW}.
On complex surfaces with Gauduchon initial metrics,
their work relates the maximal existence time to the
total volume and the areas of curves of negative
self-intersection, connecting the evolution to the
birational geometry of the underlying surface
\cite[Theorems 1.3--1.5]{TW}.

The second Chern--Ricci flow is the member of the Hermitian
curvature flow family with vanishing quadratic torsion term,
and its metric equation is strictly parabolic \cite{ST,BV}.
For this flow, Lee established short-time existence and
Shi-type estimates on complete noncompact Hermitian manifolds
under suitable boundedness and geometric hypotheses
\cite[Theorems 1.1 and 6.1]{LeeSecond}.
Bedulli--Vezzoni proved stability of Chern-flat metrics on
compact complex parallelizable manifolds.
They also showed that metrics diagonal with respect to a
global left-invariant holomorphic frame remain diagonal
and converge to Chern-flat metrics
\cite[Theorem 1.2 and Proposition 6.1]{BV}.

For K\"ahler metrics, these contractions coincide, and both
flows reduce to the K\"ahler--Ricci flow.
Outside the K\"ahler setting, their difference is ruled
by derivatives of the Chern torsion, \[
\operatorname{Ric}^{(1)}_{i\bar{j}}
-\operatorname{Ric}^{(2)}_{i\bar{j}}
=
-g^{k\bar{l}}
\left(
\nabla_{\bar{j}} T_{ik\bar{l}}
+\nabla_k \overline{T_{jl\bar{i}}}
\right).
\]
The distinction is also visible at the level of the connection:
the first contraction traces the endomorphism indices of the
Chern curvature, whereas the second traces its base indices.

These flows raise a natural question about the evolution
of holonomy.
For complete Ricci flows with bounded curvature,
Kotschwar proved preservation of restricted holonomy
\cite{Khol}, and Cook--Kotschwar established preservation
of the full holonomy group \cite{CK}.
Kotschwar's argument uses a closed homogeneous system
for curvature components and transported projections.
In the Hermitian setting, related preservation results
involve both curvature conditions and the choice of connection.
Lee's K\"ahler splitting theorem under nonpositive
bisectional curvature uses a parallel Ricci-null distribution
\cite[Corollary 7.1]{LeeSecond}.
For a different member of the Hermitian curvature flow family,
Ustinovskiy proved preservation of invariant holomorphic
subbundles and parallel sections for a torsion-twisted
connection \cite[Theorems 4.2 and 4.4]{Ust};
see also \cite{UstThesis}.

We investigate whether holonomy rigidity holds for the Chern
connection under the first and second Chern--Ricci flows,
and how torsion modifications of the connection affect
this question.
The two flows exhibit contrasting behavior.
For the first Chern--Ricci flow, our examples show that
bounded curvature and torsion do not suffice to preserve
Chern holonomy.
This failure occurs even for immortal pluriclosed solutions
with a fixed closed torsion three-form and uniform bounds
on curvature, torsion, and all their covariant derivatives.
For the second Chern--Ricci flow, we prove that full and
restricted Chern holonomy are preserved throughout every
smooth finite time interval on a compact complex manifold,
including the initial time.
The preservation theorem require no curvature sign
condition and does not assume the holonomy group
being closed.

\subsection{A contrasting preservation theorem}

The theorem below concerns the full and restricted Chern holonomy
at finite times for arbitrary compact Hermitian initial metrics.
\begin{theorem}\label{thm:intro-second}
Along a compact smooth second Chern--Ricci flow, full and restricted
Chern holonomy are constant after the time-dependent unitary
identification defined by \(U_t=\tfrac12(\Ric^{(2)})^\sharp U\).
The assertion includes the initial time and does not require the
holonomy group to be closed. Parallel holomorphic tensors made
from \(T^{1,0}M\) and its dual are fixed as actual sections.
\end{theorem}
Here \((\Ric^{(2)})^\sharp\) denotes the \(g(t)\)-self-adjoint
endomorphism of \(T^{1,0}M\) characterized by
\[
 g(t)\bigl((\Ric^{(2)})^\sharp X,\overline Y\bigr)
 =\Ric^{(2)}(t)(X,\overline Y).
\]

The proof in Section~\ref{sec:second} also applies to Hermitian bundle
heat flow over a fixed Hermitian base. Forward preservation of
principal-bundle reductions is established in \cite[Lemma 7]{BJS};
related connectedness and reduction statements appear in
\cite[Section 4]{BGO}. The Lie-group ODE principle used to compare
actual full groups is also present in \cite[Proposition 3.2 and
Lemma 3.3]{CK}. Here we implement these mechanisms for the
self-coupled second Chern--Ricci flow, use positive-time analyticity
for both time directions, and give a separate argument at the
initial time. The proof does not assume closed holonomy.
The second raised contraction takes the base-slot trace of
endomorphism-valued curvature and hence lies in the complexified
holonomy algebra. 

Section~\ref{sec:mean} applies the bundle version of this theorem
to mean curvature. The optimal lower bound of Li--Zhang--Zhang
\cite[Theorem 1.5]{LZZ} can be approached within the full and
restricted holonomy class of any prescribed initial metric,
although the optimum need not be attained in that class.

\subsection{An explicit first-flow example}
We first give a matrix-valued construction on a complex two-torus.

\begin{theorem}\label{thm:intro-chern}
For every \(\tau>0\), there is an explicit immortal first Chern--Ricci
flow of pluriclosed metrics on a complex two-torus such that its full
and restricted Chern holonomy groups are \(U(2)\) at every finite
\(t\ne\tau\), and a maximal torus \(U(1)^2\) at \(t=\tau\).
The flow converges smoothly to a metric with full and restricted
Chern holonomy \(SU(2)\). All spatial and temporal derivatives of
curvature and torsion are uniformly bounded.
\end{theorem}

The scalar ingredient is
\[
 P_s(x)=\frac{\sinh s}{\cosh s+\cos x},\qquad
 \partial_sP_s=(\log P_s)''.
\]
Fix \(S=s_0+\tau/4\), choose \(B>\max P_S\), and set
\[
 d=(B^2-BP_S)^{1/2},\qquad
 g(t)=
 \begin{pmatrix}P_{s_0+t/4}+d^2/B&d\\d&B\end{pmatrix}.
\]
Its determinant is \(BP_{s_0+t/4}\), so the flow reduces exactly to
the scalar identity. At the target time the two constant eigenlines
are parallel. At other times a curvature commutator generates the
missing non-Abelian directions. The exact calculation is in
Theorem~\ref{thm:chern-crossing}.

\subsection{Prior work and scope}
The explicit Hopf solutions originate in Tosatti--Weinkove
\cite{TW} and are written in the parametrization used here by
Yang \cite{YangHopf}. They already contain an isolated interior
decrease of Chern holonomy.
In complex dimension two, his formula (2.2), with \(T_0>0\), has
Chern holonomy \(U(2)\) except at \(t=T_0/2\), when it is the central
\(U(1)\). We supply the holonomy calculation in Appendix~\ref{app:hopf}.
Theorem~\ref{thm:intro-chern} instead gives an
immortal, uniformly nondegenerate example on a complex torus,
with all-order bounds for all time and an explicitly determined
special-unitary limit. The Hopf solution becomes singular in finite
time, although every compact interval around its crossing has
bounded curvature and torsion.
The scalar profile \(P_s\) is classical; see
\cite[Section 3.11.2, formula (3.53)]{Fox}, which also records the
corresponding immortal surface flow on a torus. General long-time
convergence for vanishing first Bott--Chern class is due to Gill
\cite{Gill}. These results already supply the long-time analytic
background on a complex torus. Likewise, the rescaling
\(\widetilde g(t)=c\,g(t/c)\), \(c>0\), moves a crossing time.
The construction here combines the required geometric conditions
with exact full-group calculations; its direct estimates provide
a self-contained verification rather than a new general
long-time existence theorem.

The main contrast is between the first-flow construction in
Section~\ref{sec:crossing} and the second-flow preservation theorem
in Section~\ref{sec:second}. Section~\ref{sec:mean} gives the
mean-curvature application. Section~\ref{sec:further} collects
the projected-curvature obstruction, examples for torsion-modified
connections, and sufficient conditions for preservation.
Appendix~\ref{app:hopf} records the Chern holonomy calculation
for the previously published Hopf solution.

\subsection*{Conventions}
All manifolds in the flow statements are connected and without
boundary. We use \(F=dA+A\wedge A\) for a connection form acting on column
coefficients. In holomorphic coordinates,
\[
 (\Gamma_i^{\mathrm{Ch}})^k{}_j
  =g^{k\bar\ell}\partial_i g_{j\bar\ell},\qquad
 R_{i\bar j}=-\partial_{\bar j}\Gamma_i^{\mathrm{Ch}}.
\]
For \(x=\Re z^1\), \(\partial_1\partial_{\bar1}=\frac14\partial_x^2\).
We compute restricted groups by Ambrose--Singer \cite{AS} and bound
full holonomy separately by a globally defined connection form.

\section{An explicit pluriclosed Chern holonomy crossing}\label{sec:crossing}
\subsection{The solution}
Let \(M=\mathbb C^2/(2\pi\mathbb Z+2\pi i\mathbb Z)^2\), with
\(z^j=x^j+iy^j\), and put \(x=x^1\), \(y=y^1\).
We use the first unnormalized Chern--Ricci flow convention
\[
 \partial_t g_{i\bar j}=\partial_i\partial_{\bar j}\log\det g.
\]
Fix \(s_0>0\) and \(\tau>0\), and define
\[
 s(t)=s_0+\tfrac14t,\quad S=s(\tau),\qquad
 P_s(x)=\frac{\sinh s}{\cosh s+\cos x}.
\]
Choose \(B>\coth(S/2)=\max_xP_S(x)\), and set
\begin{equation}\label{eq:crh-2-1}
 d(x)=\sqrt{B^2-BP_S(x)},\qquad
 G(t,x)=
 \begin{pmatrix}
 P_{s(t)}(x)+d(x)^2/B&d(x)\\
 d(x)&B
 \end{pmatrix}.
\end{equation}
All matrices below act on column vectors.
All entries in \eqref{eq:crh-2-1} are real, so no transpose convention is involved.

\begin{theorem}\label{thm:chern-crossing}
The metric \eqref{eq:crh-2-1} solves the first Chern--Ricci flow for all \(t\geq0\)
and is pluriclosed at every time.
It is uniformly equivalent to a fixed flat metric, with uniformly bounded
derivatives of every order. The full and restricted Chern holonomy groups,
up to unitary conjugacy, are
\begin{equation}\label{eq:crh-2-2}
 \Hol(g(t))=\Hol^0(g(t))=
 \begin{cases}
 U(2),&t\ne\tau,\quad t<\infty,\\
 U(1)\times U(1),&t=\tau.
 \end{cases}
\end{equation}
As \(t\to\infty\), the solution converges smoothly to
\[
 G_\infty(x)=
 \begin{pmatrix}1+d^2/B&d\\d&B\end{pmatrix},
\]
whose full and restricted Chern holonomy groups are \(SU(2)\).
\end{theorem}

Indeed, the $d(x)^2/B$ term on top left is to ensure \(\det G=BP_{s(t)}\), and the Schur complement of the lower
right entry is \(P_{s(t)}>0\). Thus \eqref{eq:crh-2-1} is positive definite. The identity
\[
 \partial_sP_s=(\log P_s)''
 =\frac{1+\cosh s\cos x}{(\cosh s+\cos x)^2}
\]
shows that \(G_t=\frac14(\log\det G)''E_{11}\), exactly the required flow.
For uniform estimates, write
\[
 P_s(x)=\frac{1-r^2}{1+2r\cos x+r^2},\qquad r=e^{-s}.
\]
The matrix \(G\) and its inverse are smooth on the compact parameter
set \(0\leq r\leq e^{-s_0}<1\), \(x\in\mathbb R/2\pi\mathbb Z\).
Since \(\partial_t=-\frac14r\partial_r\), all metric derivatives are
uniformly bounded. This also bounds the Chern torsion, curvature, and
every fixed order of their covariant and temporal derivatives.
The same argument gives smooth convergence to \(G_\infty\).
All metrics are complete because \(M\) is compact.

The pluriclosed assertion is direct. In
\(\partial\bar\partial\omega\), differentiation of any matrix entry
introduces \(dz^1\wedge d\bar z^1\). The entries in the first row or
column then give a repeated differential, while the remaining entry
\(g_{2\bar2}=B\) is constant. Thus
\(\partial\bar\partial\omega=0\).
Moreover \(d\omega(t)=d\omega(0)\), as holds for every first
Chern--Ricci flow. In particular, the Bismut torsion three-form,
up to the convention-dependent sign of \(d^c\omega\), is fixed in
time and closed. Fixed closed torsion does not prevent the Chern
holonomy crossing in this example.

\subsection{A global unitary connection formula}
For the moment put \(p=P_{s(t)}\), \(u=\log p\), and
\[
 q=\frac{d'}{\sqrt{Bp}},\qquad
 J_1=i(E_{11}-E_{22}),\quad
 J_2=E_{21}-E_{12},\quad J_3=i(E_{12}+E_{21}),\quad Z=i\Id.
\]
The periodic global unitary frame is represented in the holomorphic
coordinate frame by
\[
 F=\begin{pmatrix}p^{-1/2}&0\\-d/(B\sqrt p)&B^{-1/2}\end{pmatrix},
 \qquad F^*GF=\Id.
\]
The Chern connection in the holomorphic frame is
\(\frac12G^{-1}G'\,dz^1\). Direct multiplication gives
\[
 F^*G'F=
 \begin{pmatrix}u'&q\\q&0\end{pmatrix},\qquad
 F^{-1}F'=
 \begin{pmatrix}-u'/2&0\\-q&0\end{pmatrix}.
\]
Consequently its unitary connection form is
\begin{equation}\label{eq:crh-2-3}
 \mathcal A=-\frac q2J_2\,dx+
 \left(\frac{i u'}2E_{11}+\frac q2J_3\right)dy.
\end{equation}
There are no components in the \(x^2,y^2\) directions.
Using \([J_2,J_3]=-2J_1\), its curvature is
\begin{equation}\label{eq:crh-2-4}
 \mathcal F_{xy}=\frac{u''}{4}Z+
 \left(\frac{u''}{4}+\frac{q^2}{2}\right)J_1+
 \left(\frac{q'}2-\frac{qu'}4\right)J_3.
\end{equation}
Since all coefficients depend only on \(x\), differentiation after parallel
transport in the \(y\) direction gives
\begin{equation}\label{eq:crh-2-5}
 D_y\mathcal F_{xy}
 =[\mathcal A_y,\mathcal F_{xy}]
 = C(x)J_2,
\end{equation}
where
\begin{equation}\label{eq:crh-2-6}
 C=\frac{B(p'd''-d'p'')-2(d')^3}{4Bp\sqrt{Bp}}
   =\frac{\delta'd''-d'\delta''}{4p\sqrt{Bp}},
 \qquad \delta=p-P_S.
\end{equation}
For the second equality use \(P_S=B-d^2/B\).
The Ambrose--Singer theorem \cite{AS} implies that both
\(\mathcal F_{xy}\) and \(D_y\mathcal F_{xy}\), at any chosen base point,
belong to the restricted holonomy algebra: the latter is the derivative
of curvature transported from a short \(y\)-curve.

\subsection{The holonomy calculation}
Suppose first that \(s\ne S\). We claim that \(C\) is not identically zero.
On \(0<x<\pi\), both \(P_S'\) and \(-d'\) are positive. If \(C=0\)
identically, then
\[
 (\delta'/d')'=0,\qquad \delta=a d+c
\]
on that interval, for constants \(a,c\). Writing \(X=\cos x\),
\(P_s=\sinh s/(\cosh s+X)\) is a rational function of \(X\).
If \(a=0\), the identity \(P_s-P_S=c\) is impossible because its left
side has two distinct poles. If \(a\ne0\), squaring gives the rational
identity
\[
 (P_s-P_S-c)^2=a^2(B^2-BP_S).
\]
The left side has a double pole at \(X=-\cosh s\), whereas the right
side is regular there. This is again impossible. The claim follows.

Choose \(x\) with \(C(x)\ne0\). Equation \eqref{eq:crh-2-5} places \(J_2\) in the
restricted holonomy algebra. The traceless part of \eqref{eq:crh-2-4} is a nonzero
element of \(\operatorname{span}_{\mathbb R}\{J_1,J_3\}\); it is nonzero
because otherwise \eqref{eq:crh-2-5} would vanish. Bracketing \eqref{eq:crh-2-4} twice with \(J_2\)
isolates this traceless part, and a further bracket gives its independent
orthogonal direction. Thus the holonomy algebra contains \(\su(2)\).
Since \(P_s\) is nonconstant at every finite \(s\), the function \(u''\)
does not vanish identically. Curvature at some point therefore has
nonzero trace, and parallel transport preserves trace. The holonomy
algebra is \(\uu(2)\). The restricted holonomy is \(U(2)\), which is
also the full holonomy because the connection is unitary.

At \(s=S\), the metric is
\[
 G(\tau)=\begin{pmatrix}B&d\\d&B\end{pmatrix}.
\]
A constant holomorphic change of frame diagonalizes it to
\(\operatorname{diag}(B+d,B-d)\). Both line subbundles are therefore
Chern parallel. To show that the holonomy fills their full torus, suppose
that constants \(a,b\in\mathbb R\) satisfy
\[
 a(\log(B+d))''+b(\log(B-d))''=0.
\]
Periodicity implies \(a\log(B+d)+b\log(B-d)\) is constant. Differentiating
with respect to \(d\), which ranges over a nontrivial interval, gives
\[
 \frac{a}{B+d}-\frac{b}{B-d}=0.
\]
The constant and linear coefficients in \(d\) imply \(a=b=0\).
The curvature values therefore span the diagonal algebra \(\uu(1)^2\).
The global diagonal frame gives the matching upper bound for full
holonomy. This proves \eqref{eq:crh-2-2}.

At the limit \(p=1\), \eqref{eq:crh-2-3} is traceless, so full holonomy lies in \(SU(2)\).
Here \eqref{eq:crh-2-5} becomes
\[
 D_y\mathcal F_{xy}=-\frac{q^3}{2}J_2.
\]
Since \(d\) is nonconstant, this is nonzero somewhere. The same bracket
argument gives \(\su(2)\), completing the proof of Theorem~\ref{thm:chern-crossing}.

\subsection{Transversality and parallel tensors}
The disappearance of the non-Abelian directions at \(\tau\) is transverse.
Differentiating \eqref{eq:crh-2-6} at that time gives
\begin{equation}\label{eq:crh-2-7}
 \left.\partial_t C\right|_{\tau}
 =\frac{
  (\partial_sP_S)'d''-d'(\partial_sP_S)''
 }{16P_S\sqrt{BP_S}}.
\end{equation}
This function is not identically zero. Otherwise the interval argument
above would give \(\partial_sP_S=a d+c\). If \(a=0\), the double pole of
\[
 \partial_sP_S=\frac{1+\cosh S\,X}{(\cosh S+X)^2}
\]
is already a contradiction. If \(a\ne0\), squaring would equate a rational
function with a fourth-order pole at \(X=-\cosh S\) to one with a simple
pole there. Thus \eqref{eq:crh-2-7} is nonzero at some point.

The complex dimension of the global Chern-parallel endomorphism space is
\(1\) for every finite \(t\ne\tau\), and \(2\) at \(t=\tau\).
At \(t=\tau\) the two constant holomorphic eigenlines acquire parallel
orthogonal projections. At the limit the endomorphism dimension returns
to \(1\), while a parallel complex volume form appears. There are no
nonzero parallel vectors at any of these times. These statements concern
parallel tensors, rather than isometry groups.

\begin{remark}
The limiting metric has vanishing first Chern--Ricci tensor and
nonzero Chern curvature. This distinction also applies to the
convergence theorem of \cite{Gill}: a first-flow limit on a complex
parallelizable manifold need not be Chern-flat. Its nonzero Chern
torsion permits holonomy \(SU(2)\) on a complex torus. Finite-time
appearance of Chern determinant-flatness would instead force stationarity
on a compact smooth time slab by backward uniqueness for the first
Chern--Ricci flow. The example reaches determinant-flatness only
asymptotically.
\end{remark}


\section{Preservation under the second Chern--Ricci flow}\label{sec:second}

For possible extension, we take the following conventions. Let \(E\to M\) be a fixed holomorphic bundle over a compact complex
manifold, and let \(h(t)\) be a smooth family of Hermitian metrics.
Write the Chern curvature as
\(F_h=R_{i\bar j}\,dz^i\wedge d\bar z^j\), with
\(R_{i\bar j}=-\partial_{\bar j}(h^{-1}\partial_i h)\).
For a positive Hermitian contravariant base tensor \(b^{i\bar j}(t)\), set
\begin{equation}\label{eq:crh-7-1}
 S(h,b)=b^{i\bar j}R_{i\bar j},\qquad h^{-1}h_t=-S(h,b).
\end{equation}
The endomorphism \(S\) is \(h\)-self-adjoint.
This is the unnormalized Hermitian bundle heat equation.
For \(E=T^{1,0}M\), \(h=g\), and \(b=g^{-1}\), it is the
second Chern--Ricci flow
\begin{equation}\label{eq:crh-7-2}
 g_t=-\Ric^{(2)},\qquad
 \Ric^{(2)}_{i\bar j}=g^{k\bar l}R_{k\bar l i\bar j}.
\end{equation}
Equation \eqref{eq:crh-7-2} is strictly parabolic; see \cite{ST,BV} and
\cite[Lemma 3.1]{LeeSecond}.
It is different from first Chern--Ricci flow.

Let \(U(t):(E,h(0))\to(E,h(t))\) solve
\begin{equation}\label{eq:crh-7-3}
 U_t=\tfrac12 S\,U,\qquad U(0)=\Id .
\end{equation}
It is a complex bundle isometry. We identify all Chern connections
with unitary connections \(A(t)=U(t)^*\nabla^{h(t)}\) on \((E,h(0))\).

\begin{theorem}\label{sr-thm:main}
For a smooth second Chern--Ricci flow on a compact complex manifold,
both \(\Hol_p(A(t))\) and \(\Hol_p^0(A(t))\) are independent of time
on every smooth finite time interval, including its initial time.
The same conclusion holds for \eqref{eq:crh-7-1}, provided \(h(t)\) and the base
contraction are real-analytic in positive time with values in every
finite spatial \(C^k\) space. In particular it applies to a fixed
smooth Hermitian base metric.
\end{theorem}
The statement concerns the actual full group, including discrete
monodromy, not just its Lie algebra or its closure.
No closedness of the holonomy subgroup is assumed.

\subsection{A connection comparison lemma}
For a unitary connection \(A_0\), its restricted holonomy Lie algebra
defines an \(A_0\)-parallel real vector subbundle
\[
 \mathfrak h\subset\mathfrak u(E).
\]
Its fibers are closed vector subspaces, even when the corresponding
connected immersed holonomy groups are not closed.
Its curvature lies in \(\mathfrak h\) by Ambrose--Singer.
The next comparison uses the Lie-subgroup ODE principle of
\cite[Lemma 3.3]{CK}, applied along a spatial path.

\begin{lemma}\label{sr-lem:comparison}
If \(A_1-A_0\in\Omega^1(\mathfrak h)\), then
\[
 \Hol_p(A_1)\subset\Hol_p(A_0),\qquad
 \Hol_p^0(A_1)\subset\Hol_p^0(A_0).
\]
Here \(\mathfrak h\) is the holonomy Lie-algebra bundle of the reference
connection \(A_0\). The hypothesis need not hold after interchanging
\(A_0\) and \(A_1\).
\end{lemma}
\begin{proof}
For a based path \(\gamma\), let \(P_i(s)\) be its two parallel
transports, and put \(V(s)=P_0(s)^{-1}P_1(s)\).
Then \(V(0)=\Id\) and
\[
 V'=-P_0^{-1}(A_1-A_0)(\dot\gamma)P_0\,V .
\]
Since \(\mathfrak h\) is \(A_0\)-parallel, the coefficient belongs
to the fixed Lie algebra \(\mathfrak h_p\). The solution remains
in its connected immersed matrix Lie group \(H_p^0\).
This last assertion does not require closedness: integrate the
same finite-dimensional ODE in the connected Lie group; bounded
Lie-algebra coefficients on a compact path interval prevent escape
in a complete left- or right-invariant metric.
For a loop,
\(P_1(1)=P_0(1)V(1)\in\Hol_p(A_0)\).
If the loop is contractible both factors lie in \(H_p^0\).
\end{proof}

\begin{lemma}[Analytic affine invariance]\label{sr-lem:analytic}
Let \(A(t)\) be an analytic family of unitary connections on a
connected time interval \(I\), in every finite spatial \(C^k\)
topology. Suppose it solves \(A_t=\Phi_t(A)\), where the following
holds: whenever \(\mathfrak h\) is the holonomy Lie-algebra bundle
of a fixed connection \(A_*\), the operator \(\Phi_t\) maps
\[
 A_*+\Omega^1(\mathfrak h)\quad\hbox{into}\quad\Omega^1(\mathfrak h)
\]
for every \(t\). Assume enough smooth differential dependence to
differentiate the equation to arbitrary order.
Then the full and restricted holonomy groups of \(A(t)\), in the
fixed unitary bundle, are constant.
\end{lemma}
\begin{proof}
Fix \(\tau\in I\), take \(A_*=A(\tau)\), and let \(\mathfrak h\)
be its holonomy algebra bundle.
The affine invariance and the chain rule show inductively that
\(\partial_t^j A(\tau)\in\Omega^1(\mathfrak h)\) for every \(j\ge1\).
Use the fixed orthogonal projection \(Q\) onto \(\mathfrak h^\perp\).
Analyticity gives \(Q(A(t)-A(\tau))=0\) near \(\tau\), and the
analytic identity theorem extends this to all \(I\).
Lemma~\ref{sr-lem:comparison} gives
\(\Hol(A(t))\subset\Hol(A(\tau))\), and similarly for restricted
groups. Interchange \(t,\tau\) to obtain equality.
\end{proof}

\subsection{The evolving unitary connection}
Variation of the Chern connection in a fixed holomorphic frame gives
(compare the proof of \cite[Lemma 3.3]{LeeSecond})
\[
 (\nabla^h)^{1,0}_t=-\partial_{\nabla^h}S,\qquad
 (\nabla^h)^{0,1}_t=0.
\]
To obtain the evolution in the fixed unitary bundle, write
\(\Gamma(t)\) for the connection form of \(\nabla^{h(t)}\) in a
fixed holomorphic frame. The pullback connection has local form
\[
 A=U^{-1}\Gamma U+U^{-1}dU.
\]
Differentiating this identity gives
\[
 A_t=U^{-1}\Gamma_tU+d_A(U^{-1}U_t).
\]
By \eqref{eq:crh-7-3}, \(U^{-1}U_t=\tfrac12\widetilde S\), where
\(\widetilde S=U^{-1}SU\). Since
\(U^{-1}\Gamma_tU=-\partial_A\widetilde S\), we obtain
\begin{equation}\label{eq:crh-7-4}
 A_t=\tfrac12\bigl(\bar\partial_A-\partial_A\bigr)\widetilde S,
 \qquad
 \widetilde S=U^{-1}SU=b^{i\bar j}(R_A)_{i\bar j}.
\end{equation}
The base type decomposition is fixed, so \eqref{eq:crh-7-4} is intrinsic even though
the pulled-back bundle Dolbeault operator varies.

We now verify the hypotheses of Lemma~\ref{sr-lem:analytic}
for this equation and explain how its conclusion gives
Theorem~\ref{sr-thm:main} at positive times.

\smallskip
\noindent\emph{The operator and its affine invariance.}
Regard \(b(t)\) as a prescribed family of base coefficients.
For the second flow, this is the family \(g(t)^{-1}\) belonging
to the given solution; when testing the operator on other
connections, these coefficients are kept fixed.
For an arbitrary \(h(0)\)-unitary connection \(B\), define
\[
 S_t(B)=b^{i\bar j}(t)(F_B)_{i\bar j},\qquad
 \Phi_t(B)=\tfrac12(\bar\partial_B-\partial_B)S_t(B).
\]
Here \((F_B)_{i\bar j}\) denotes the mixed curvature component.
Equation \eqref{eq:crh-7-4} is exactly \(A_t=\Phi_t(A)\).
The operator \(\Phi_t\) is a smooth differential operator of
order two in \(B\); locally it depends on the connection
coefficients and their first two spatial derivatives, as well
as on the prescribed coefficients \(b(t)\) and their derivatives.
Thus its dependence has the differentiability required in
Lemma~\ref{sr-lem:analytic}.

Fix a unitary connection \(A_*\) and let \(\mathfrak h\) be its
holonomy Lie-algebra bundle. If \(B=A_*+a\) with
\(a\in\Omega^1(\mathfrak h)\), then
\[
 F_B=F_{A_*}+d_{A_*}a+a\wedge a
       \in\Omega^2(\mathfrak h).
\]
Indeed, \(F_{A_*}\) is \(\mathfrak h\)-valued by Ambrose--Singer,
\(d_{A_*}a\) is \(\mathfrak h\)-valued because the subbundle is
\(A_*\)-parallel, and \(a\wedge a\) is \(\mathfrak h\)-valued
because its fibers are closed under Lie brackets.
Moreover, for a section \(w\) of \(\mathfrak h\),
\[
 \nabla_Bw=\nabla_{A_*}w+[a,w]\in\Omega^1(\mathfrak h).
\]
Hence \(B\) also preserves \(\mathfrak h\) and its complexification.
It follows successively that
\[
 S_t(B)\in\Gamma(\mathfrak h_{\mathbb C}),\qquad
 \Phi_t(B)\in\Omega^1(\mathfrak h_{\mathbb C}).
\]
Since \(B\) is unitary and \(b\) is Hermitian, \(S_t(B)\) is
self-adjoint. On real tangent vectors,
\((\bar\partial_B-\partial_B)S_t(B)\) is skew-adjoint.
Consequently
\[
 \Phi_t(B)\in\Omega^1(\mathfrak h_{\mathbb C}\cap\mathfrak u(E))
             =\Omega^1(\mathfrak h).
\]
This is precisely the affine invariance in Lemma~\ref{sr-lem:analytic}.
The time dependence of \(b(t)\) causes no additional Lie-algebra
directions: its coefficients and all their derivatives act as
scalars on the endomorphism factor.

\smallskip
\noindent\emph{Analyticity of the actual connection family.}
For \eqref{eq:crh-7-2}, the local metric equation has principal term
\[
 g^{k\bar l}\partial_k\partial_{\bar l}g_{i\bar j}
\]
and lower-order terms analytic in \(g,g^{-1},\partial g\).
The fixed-base bundle equation has the same structure with
\(b^{k\bar l}\) in place of \(g^{k\bar l}\). Write either of
these autonomous metric equations as \(q_t=\mathcal Q(q)\).
Fix \(t_*>0\) and a short slab \([s,s+\delta]\) containing
\(t_*\) in its interior, where \(s>0\). On parabolic H\"older spaces,
the map
\[
 (\lambda,q)\longmapsto
 \bigl(q_t-\lambda\mathcal Q(q),\ q(s)-q_{\mathrm{given}}(s)\bigr)
\]
is analytic near the given solution and \(\lambda=1\).
For example, for each integer \(k\ge0\) and \(0<\alpha<1\),
its solution space is \(C^{k+2+\alpha,(k+2+\alpha)/2}\), and its
target is the product of \(C^{k+\alpha,(k+\alpha)/2}\) with
the initial-trace space \(C^{k+2,\alpha}\).
The derivative in \(q\) is
\[
 w\longmapsto
 \bigl(w_t-D\mathcal Q_{q(t)}w,\ w(s)\bigr),
\]
a bounded isomorphism by the linear parabolic Schauder theorem.
The analytic implicit function theorem therefore makes the
solution \(q_\lambda\) analytic in \(\lambda\). Uniqueness gives
\(q_\lambda(t)=q(s+\lambda(t-s))\), after shortening the slab
so that all rescaled times remain in the original solution interval.
Evaluation at \(t_*\) proves time analyticity in every finite
spatial H\"older norm; see \cite{Angenent} for this parameter method.
This argument requires smooth, rather than spatially analytic,
fixed coefficients.

Thus \(h(t)\), \(b(t)\), and \(S(t)\) are analytic at positive
times for the second flow and for fixed-base bundle heat.
For the general bundle family in Theorem~\ref{sr-thm:main},
this follows from its stated analyticity hypotheses.
The linear ODE \eqref{eq:crh-7-3} then implies that \(U(t)\)
is analytic at positive times in every finite spatial \(C^k\) norm.
The formula \(A=U^{-1}\Gamma U+U^{-1}dU\) shows that \(A(t)\)
has the same property. All these connections are unitary for
the single fixed metric \(h(0)\).

\smallskip
\noindent\emph{Theorem 3.1 from Lemma~\ref{sr-lem:analytic}.}
Apply the lemma on the connected interval \(I=(0,T)\).
To make its group conclusion explicit, fix \(\tau\in I\) and
write \(\mathfrak h_\tau\) for the holonomy Lie-algebra bundle
of \(A(\tau)\). The affine-invariance argument in the lemma
puts every time derivative \(\partial_t^jA(\tau)\), \(j\ge1\),
in \(\Omega^1(\mathfrak h_\tau)\). If \(Q_\tau\) denotes
the fixed orthogonal projection onto \(\mathfrak h_\tau^\perp\),
then the analytic function \(Q_\tau(A(t)-A(\tau))\) has all
derivatives zero at \(\tau\). It vanishes near \(\tau\), and
the analytic identity theorem extends this equality to all of \(I\).
In particular,
\[
 A(t)-A(\tau)\in\Omega^1(\mathfrak h_\tau)
 \qquad(t,\tau>0).
\]
Lemma~\ref{sr-lem:comparison} now gives actual group inclusions
\[
 \Hol_p(A(t))\subset\Hol_p(A(\tau)),\qquad
 \Hol_p^0(A(t))\subset\Hol_p^0(A(\tau)).
\]
Repeating the same argument with \(t\) and \(\tau\) interchanged
gives both reverse inclusions, hence equality of the full groups
and of the restricted groups. Finally,
\[
 \Hol_p(A(t))
 =U_p(t)^{-1}\Hol_p(\nabla^{h(t)})U_p(t),
\]
with the same identity for restricted holonomy. This proves
Theorem~\ref{sr-thm:main} at positive times under the stated
unitary identification. The next subsection supplies the two
inclusions needed to include \(t=0\).

\subsection{Including the initial time}
We give the forward reduction argument explicitly, to avoid deducing
initial-time equality from continuity alone.

\begin{lemma}\label{sr-lem:forward}
For the bundle metric equation \eqref{eq:crh-7-1}, with any prescribed smooth
positive base contraction \(b(t)\), the initial full holonomy is
an upper bound for the holonomy of the pulled-back connection \(A(t)\)
on a sufficiently short time interval.
\end{lemma}
\begin{proof}
Let \(h_0=h(0)\), \(\nabla_0=\nabla^{h_0}\), and let
\(\mathfrak h_0\) be its holonomy Lie-algebra bundle.
For a small \(h_0\)-self-adjoint endomorphism \(v\), put
\(h=h_0e^v\), and define
\[
 \mathcal E_v(w)=e^{-v}(d\exp)_v(w)
               =\int_0^1 e^{-sv}w e^{sv}\,ds.
\]
For \(v\) sufficiently small this is invertible.
If \(v\) is a section of \(i\mathfrak h_0\), then
\begin{equation}\label{eq:crh-7-5}
 \nabla^{h,1,0}
 =\nabla_0^{1,0}+e^{-v}\partial_{\nabla_0}e^v .
\end{equation}
Since \(\mathfrak h_{0,\mathbb C}\) is \(\nabla_0\)-parallel and
closed under commutators, \eqref{eq:crh-7-5} has
\(\mathfrak h_{0,\mathbb C}\)-valued difference, and its curvature
is likewise \(\mathfrak h_{0,\mathbb C}\)-valued.
Thus \(S(h,b)\in\mathfrak h_{0,\mathbb C}\).

The logarithmic metric equation is
\begin{equation}\label{eq:crh-7-6}
 v_t=-\mathcal E_v^{-1}S(h_0e^v,b).
\end{equation}
The operator \(\mathcal E_v\) and its local inverse preserve
\(\mathfrak h_{0,\mathbb C}\), since they are analytic functions
of \(\operatorname{ad}_v\). Moreover,
\(S(h,b)\) is self-adjoint for \(h=h_0e^v\). Thus
\(-S(h,b)\) represents a tangent vector to the space of positive
Hermitian metrics; applying the differential of the logarithm
makes \(-\mathcal E_v^{-1}S(h_0e^v,b)\) self-adjoint for \(h_0\).
The right side of \eqref{eq:crh-7-6} therefore lies in
\[
 \mathfrak h_{0,\mathbb C}\cap
       \{h_0\text{-self-adjoint endomorphisms}\}=i\mathfrak h_0.
\]
Its principal part is
\(b^{i\bar j}\partial_i\partial_{\bar j}v\);
the two factors \(\mathcal E_v\) cancel in the highest derivatives.
Replacing coordinate derivatives by covariant derivatives of the
induced connection \(\nabla_0\) changes only lower-order terms.
In real coordinates, the principal operator can thus be written as
\(a^{\alpha\beta}\nabla_{0,\alpha}\nabla_{0,\beta}v\), where
\(a^{\alpha\beta}\) is the positive definite real principal symbol
of \(b^{i\bar j}\partial_i\partial_{\bar j}\).
Since \(\nabla_0\) preserves \(i\mathfrak h_0\), this covariant
principal operator acts on sections of that subbundle.
It is therefore a strictly parabolic equation on the real vector
bundle \(i\mathfrak h_0\).
The subbundle \(i\mathfrak h_0\) is a smooth real vector bundle.
In local frames, restriction of \eqref{eq:crh-7-6} has the same scalar
uniformly parabolic principal symbol and smooth lower-order
coefficients. Standard short-time theory therefore gives a
solution in this subbundle. It is also a solution of the full
metric equation, whose uniqueness shows that the actual
solution, initially \(v=0\), takes values in \(i\mathfrak h_0\).

Since \(S(t)\in\mathfrak h_{0,\mathbb C}\), the ODE \eqref{eq:crh-7-3} shows
that conjugation by \(U(t)\) preserves the complex Lie-algebra subbundle
\(\mathfrak h_{0,\mathbb C}\).
This is a linear ODE on endomorphisms, using closure under brackets.
There is also a direct derivative identity
\[
 \partial_t\bigl(U^{-1}\nabla_0 U\bigr)
       =\tfrac12 U^{-1}(\nabla_0S)U.
\]
The right side takes values in \(\mathfrak h_{0,\mathbb C}\)
because the subbundle is \(\nabla_0\)-parallel, and the left
quantity starts at zero. Thus \(U^{-1}\nabla_0U\) is valued in
the same subbundle. Gauge-transforming \eqref{eq:crh-7-5} now gives
\(A(t)-A(0)\in\Omega^1(\mathfrak h_{0,\mathbb C})\).
As both connections are \(h_0\)-unitary, this difference belongs
to \(\Omega^1(\mathfrak h_0)\).
Lemma~\ref{sr-lem:comparison} proves the assertion. The gauge step
uses only a closed vector subbundle, and makes no closed-group assumption.
\end{proof}

In the self-coupled equation \eqref{eq:crh-7-2}, apply this lemma with the actual
smooth coefficients \(b(t)=g(t)^{-1}\) regarded as prescribed.
The actual metric and the restricted solution solve the same bundle
metric equation with these coefficients, so uniqueness still applies.

To finish Theorem~\ref{sr-thm:main}, fix \(\tau>0\).
The positive-time argument gives
\(A(t)-A(\tau)\in\Omega^1(\mathfrak h_\tau)\) for all \(t>0\).
Taking \(t\downarrow0\), smoothness of \(A\) and closedness of
the finite-dimensional fibers of \(\mathfrak h_\tau\) give
\[
 A(0)-A(\tau)\in\Omega^1(\mathfrak h_\tau).
\]
Applying Lemma~\ref{sr-lem:comparison} with reference connection
\(A(\tau)\), we obtain
\[
 \Hol_p(A(0))\subset\Hol_p(A(\tau)),\qquad
 \Hol_p^0(A(0))\subset\Hol_p^0(A(\tau)).
\]
For the reverse inclusions, choose a sufficiently small
\(\varepsilon>0\) in the interval of Lemma~\ref{sr-lem:forward}.
That lemma and its connection-difference proof give
\(\Hol_p(A(\varepsilon))\subset\Hol_p(A(0))\), also for
restricted holonomy. Positive-time constancy gives
\[
 \Hol_p(A(\tau))=\Hol_p(A(\varepsilon))\subset\Hol_p(A(0)),
\]
and the same chain for \(\Hol_p^0\). Combining the two directions
proves the theorem including the initial time. The limiting step
uses closed vector spaces in the Lie-algebra bundle; no closedness
of the holonomy groups is needed.

\begin{corollary}[Actual holomorphic parallel tensors]
For the flows in Theorem~\ref{sr-thm:main}, the space of parallel
sections in every tensor bundle formed from \(E\) and \(E^*\)
using complex tensor products, symmetric powers, or exterior powers
is constant as a space of actual sections. The statement also
holds on the universal cover. Conjugate tensor factors are excluded
from this assertion of actual, rather than gauge-transported, equality.
\end{corollary}
\begin{proof}
In the fixed unitary bundle, \(A(t)-A(\tau)\) takes values in the
holonomy Lie algebra of \(A(\tau)\). It therefore annihilates every
\(A(\tau)\)-parallel tensor, so the same tensor is parallel for
every \(A(t)\). Undoing \(U(t)\) gives a parallel tensor \(s(t)\)
for \(\nabla^{h(t)}\), with
\(s_t=\tfrac12\rho(S)s\), where \(\rho\) is the infinitesimal
complex tensor representation. Curvature annihilates a parallel
tensor, hence \(\rho(S)s=b^{i\bar j}\rho(R_{i\bar j})s=0\).
Thus \(s(t)\) is actually independent of time. Lift the same
connection-difference identity to the universal cover for the
corresponding assertion there.
\end{proof}

\begin{remark}[The group and the holonomy of a fixed loop]
Constancy of the group does not identify individual loop
transports at different times. On the flat elliptic curve
\(X=\mathbb C/(2\pi\mathbb Z+2\pi i\mathbb Z)\), take the
trivial line-bundle metric \(h=e^f\), where
\[
 f(t,x)=e^{-t/4}\cos x+e^{-t}\cos(2x).
\]
With \(g_{z\bar z}=1\), bundle heat is \(f_t=f_{xx}/4\),
so this is an immortal solution. Its Chern form is
\(A=\partial f=(f_x/2)\,dz\). At a point with
\(\cos x_*=-1/4\) and \(0<x_*<\pi\), the vertical loop has
\[
 P_\gamma(t)=\exp(-i\pi f_x(t,x_*)),\qquad
 f_x(t,x_*)=\frac{\sqrt{15}}4(e^{-t}-e^{-t/4}).
\]
It has transport \(1\) at \(t=0\) and nonidentity transport
at small positive time. The constant loop at the same point
has transport \(1\) at every time. Thus the rule
\(P_\gamma(0)\mapsto P_\gamma(t)\) need not be well-defined.
Nevertheless \(f_{xx}\not\equiv0\) at every finite time, so
full and restricted holonomy are always \(U(1)\).
The comparison lemma above requires no such identification
of loop representatives.
\end{remark}

\subsection{An explicit example with nonclosed restricted holonomy}
On a complex two-torus, let
\[
 P_s(x)=\frac{\sinh s}{\cosh s+\cos x},\quad s=s_0+t/4,
 \qquad g(t)=\diag(P_s,P_s^\alpha),
\]
where \(s_0>0\) and \(\alpha\) is irrational.
Writing \(u=\log P_s\), the identity
\(\partial_sP_s=u''\) gives
\(u_t=e^{-u}u''/4\). Both diagonal entries therefore satisfy
the second Chern--Ricci equation \eqref{eq:crh-7-2}.
In the global unitary frame,
\[
 A=\tfrac i2\diag(1,\alpha)u'\,dy^1.
\]
Since \(u''\not\equiv0\) at every finite time, full and restricted
Chern holonomy are exactly
\[
 \{\diag(e^{i\theta},e^{i\alpha\theta}):\theta\in\mathbb R\}.
\]
This one-dimensional immersed group is dense in \(U(1)^2\).
The metric converges smoothly, with uniform all-order bounds,
to a flat metric as \(t\to\infty\). Its limiting holonomy is trivial.
The preservation and convergence of this diagonal ansatz fit the
general result of \cite[Proposition 6.1]{BV}; the calculation above
specifies its exact nonclosed holonomy at finite times.
Thus the theorem preserves the actual nonclosed group at finite
times and makes no assertion about equality with the limiting group.

\section{Positive mean curvature with prescribed holonomy}\label{sec:mean}
This section applies Theorem~\ref{sr-thm:main} to the following
question: does requiring a metric to retain a prescribed Chern
holonomy group restrict the best lower bound for its mean curvature?
The base metric is fixed, and only the metric on the holomorphic
bundle evolves. The two corollaries below show that the optimal
value can still be approached within every realizable holonomy class.

Let \((M,\omega)\) now be compact connected K\"ahler and let \(E\) be a
holomorphic bundle. Define degree using
\(\int_M c_1(E)\wedge\omega^{n-1}/(n-1)!\), put
\(V=\int_M\omega^n/n!\), and let \(\mu_{\min}(E,\omega)\)
be its minimal Harder--Narasimhan slope. Set
\(\nu_{\min}=2\pi\mu_{\min}/V\).

\begin{corollary}\label{sr-cor:mean-positive}
If \(\mu_{\min}(E,\omega)>0\), then from any prescribed smooth
initial metric \(h_0\), the fixed-base bundle heat flow eventually
has strictly positive mean curvature
\(S(t)=\sqrt{-1}\Lambda_\omega F_{h(t)}>0\).
At every finite time its full and restricted Chern holonomy are
conjugate to those of \(h_0\). In particular, every ample bundle
has a positive-mean-curvature metric in the same full Chern holonomy
class as any prescribed initial metric.
\end{corollary}
\begin{proof}
The bundle heat flow exists smoothly for every finite time on a
compact K\"ahler base. Under the unitary identification of
Section~\ref{sec:second}, the K\"ahler identities turn
\eqref{eq:crh-7-4} into
\[
 A_t=-\tfrac12 d_A^*F_A.
\]
Thus \(\mathcal A(\sigma)=A(2\sigma)\) solves the Yang--Mills
flow in the normalization of \cite[Section 3.1]{Sibley}; the
identification conjugates the mean curvature endomorphism and
preserves its eigenvalues. The standard Harder--Narasimhan asymptotics
give a sequence \(t_j\to\infty\) on which the eigenvalues of \(S(t_j)\)
converge in finite \(L^p\) to the corresponding Einstein constants
of the HN type; see \cite[Lemma 3.3 and Theorem 5.13]{Sibley}.
Unitary gauge changes do not affect these eigenvalues, and
eigenvalues of Hermitian endomorphisms depend Lipschitz
continuously on the endomorphism. The singular set in the
Uhlenbeck limit has measure zero; the uniform bound on \(S\)
and the cited \(L^p\) convergence give convergence on \(M\).
In particular, if \(\ell(t,x)\) is the smallest eigenvalue, then
\(\ell(t_j,\cdot)\to\nu_{\min}\) in \(L^1\).

In the unitary gauge the mean curvature satisfies a covariant heat
equation. The two orders of complex covariant differentiation have
the same contraction on \(S\), since their commutator is \([S,S]=0\).
The smallest-eigenvalue maximum principle therefore gives
\((\partial_t-\Delta_\omega)\ell\ge0\) in the barrier sense,
with the positive normalization of the scalar Laplacian dictated
by \eqref{eq:crh-7-1}. Hence
\(w=(\nu_{\min}-\ell)_+\) is a nonnegative scalar subsolution.
The heat kernel at time one on the fixed compact base gives
\[
 \|w(t_j+1)\|_{L^\infty}
       \le C_\omega\|w(t_j)\|_{L^1}\longrightarrow0.
\]
Also \(\min_M\ell(t)\) is nondecreasing. Thus for all sufficiently
large \(t\), \(\ell(t)\ge\nu_{\min}/2>0\).
Theorem~\ref{sr-thm:main} preserves full and restricted holonomy at
every finite time, proving the first assertion.
Ampleness implies positive minimal slope for every K\"ahler base
metric, as in \cite[Corollary 4.6]{LZZ}, giving the last assertion.
\end{proof}

This does not offer a new proof of all the HN asymptotic theory.
The existence of a positive-mean-curvature metric for an ample bundle
is already known, even for any prescribed Hermitian base metric,
by \cite[Theorem 1.10]{LZZ}. The additional formulation here is the
prescribed finite-time holonomy class and its realization from an
arbitrary starting metric over a K\"ahler base. Mean positivity
does not imply Griffiths semipositivity, and no such implication
is used.

The extreme-eigenvalue limit used above is also established directly
by Chen--Li--Zhang--Zhang \cite[Theorem 1.4, formula (1.13)]{CLZZ},
on a compact Gauduchon base. Their normalization
\(H^{-1}H_\sigma=-2(S_H-\lambda\Id)\) becomes
\(h^{-1}h_t=-S_h\) after setting
\(h(t)=e^{-\lambda t}H(t/2)\); the scalar factor does not change
the Chern connection. We retain the K\"ahler formulation here.

\begin{corollary}[The optimal mean bound in each holonomy class]
\label{sr-cor:optimal-class}
For any initial metric \(h_0\), let \(\mathcal C(h_0)\) be the set
of smooth Hermitian metrics such that, at one fixed base point,
a single complex-linear isometry simultaneously conjugates their
full and restricted Chern holonomy groups to those of \(h_0\).
Then
\begin{equation}\label{eq:optimal-holonomy-class}
 \sup_{h\in\mathcal C(h_0)}\;
       \min_{x\in M}\lambda_{\min}
          (\sqrt{-1}\Lambda_\omega F_h(x))
       =\frac{2\pi}{V}\mu_{\min}(E,\omega).
\end{equation}
There is no positivity assumption on the slope in this statement.
If \(E\) is stable, this supremum is attained in \(\mathcal C(h_0)\)
if and only if that class contains its Hermitian--Einstein metric.
\end{corollary}
\begin{proof}
The Chern--Weil inequality for the least-slope quotient gives the
displayed number as an upper bound for every metric; equivalently,
use the upper-bound part of \cite[Theorem 1.5]{LZZ}.
The proof of Corollary~\ref{sr-cor:mean-positive}, without requiring
\(\nu_{\min}>0\), gives
\(\min_M\ell(t)\to\nu_{\min}\) along bundle heat flow from \(h_0\).
Every finite-time metric belongs to \(\mathcal C(h_0)\), so the
same upper bound is the supremum within that class.

If \(E\) is stable, \(\mu_{\min}=\mu(E)\).
An attaining metric satisfies \(S\ge\nu_{\min}\Id\).
The Chern--Weil trace integral is exactly
\(\operatorname{rank}(E)\nu_{\min}V\), so the nonnegative
endomorphism \(S-\nu_{\min}\Id\) has identically zero trace.
It vanishes, and the metric is Hermitian--Einstein.
Uniqueness on a stable bundle makes this metric a constant multiple
of the Hermitian--Einstein metric. The converse is immediate.
\end{proof}

The unrestricted supremum is already \cite[Theorem 1.5]{LZZ}.
The added conclusion here is that imposing a finite-time holonomy
class does not change that value, although it can prevent attainment.
For a stable bundle, any smooth limit of optimizing metrics in a
different holonomy class must therefore leave that class.
The assertion of nonattainment does not itself assert smooth
compactness of arbitrary optimizing sequences.

\section{Further aspects of first Chern--Ricci flow}\label{sec:further}
We record consequences that are directly related to the first-flow
examples: the failure of the projected-curvature argument,
holonomy changes for torsion-modified connections, and initial
conditions under which the Chern connection is preserved.

\subsection{Why a homogeneous projected-curvature system fails}\label{sec:error}
The examples already rule out a general holonomy-preservation theorem.
A diagonal companion makes the failure of the usual error system
especially transparent.
On a complex three-torus use the same scalar profile as in Section~\ref{sec:crossing},
put \(b=P_S\), and set
\[
 \widehat g(t)=\diag(P_{s_0+t/4},b,b^{-1}).
\]
Its determinant is \(P_{s_0+t/4}\), so it is an immortal
Chern--Ricci solution with uniform all-order bounds.
It is balanced at every time: in \(d\widehat\omega^{\,2}\), every
term containing \(dz^1\wedge d\bar z^1\) vanishes upon differentiation,
and the remaining coefficient is \(b\,b^{-1}=1\).
Write \(u=\log P_{s_0+t/4}\), \(v=\log b\), and \(q_s=(\log P_s)''\).
In its global unitary frame,
\[
 A=\tfrac i2\diag(u',v',-v')\,dy^1,\qquad
 F=\tfrac i2\diag(q_s,q_S,-q_S)\,dx\wedge dy^1.
\]
At \(\tau\), full and restricted holonomy are the circle generated
by \(iL\), where \(L=\diag(1,1,-1)\).

On skew-Hermitian endomorphisms use \(\langle A,B\rangle=-\tr(AB)\).
Let \(Q\) be orthogonal projection onto the complement of
\(\mathbb R iL\), and put \(D=E_{11}-L/3\).
In the temporal gauge given by the displayed unitary frames, \(Q\)
is fixed and parallel at all times: all connection matrices commute
with \(L\). Consequently
\begin{equation}\label{eq:QF-new}
 QF=\tfrac i2(q_s-q_S)D\,dx\wedge dy^1.
\end{equation}
At \(t=\tau\), this tensor and all its spatial derivatives vanish.
However \(q_s(0)=1/(\cosh s+1)\), so
\begin{equation}\label{eq:transverse-new}
 \left.\partial_t(QF)(\partial_x,\partial_{y^1})\right|_{(\tau,0)}
 =-\frac{i\sinh S}{8(\cosh S+1)^2}D\ne0.
\end{equation}
Any covariant time-connection term acting on \(QF\) vanishes there,
so the nonzero derivative is independent of that notational choice.

For example, consider
\[
 X=(QF,Q\nabla^{\mathrm{Ch}}F),\qquad
 Y=(\nabla^{\mathrm{Ch}}Q,(\nabla^{\mathrm{Ch}})^2Q).
\]
At the terminal time \(X=Y=0\), and all spatial derivatives of \(X\)
vanish. Equation \eqref{eq:transverse-new} therefore excludes an estimate
\[
 |(\partial_t-\Delta)X|
 \le C(|X|+|\nabla X|+|Y|)
\]
with a finite uniform \(C\), even if it is initially asserted only
for \(t<\tau\) and then passed to the terminal limit.
The issue occurs before applying backward uniqueness: these errors
do not satisfy a closed homogeneous system.

The missing geometric identity is visible in the two Chern contractions.
At \(\tau\), with \(q=q_S\), their raised endomorphisms are
\[
 S^{(1)}=-\frac{q}{4b}E_{11},\qquad
 S^{(2)}=-\frac{q}{4b}L.
\]
The second follows the curvature endomorphism direction; the first,
which drives the metric flow, has a component transverse to it.
Indeed,
\[
 R_{1\bar1\,2\bar2}=-\tfrac14bq,\qquad
 R_{2\bar2\,1\bar1}=0.
\]
The slot-exchange symmetry of K\"ahler curvature fails.
The same point has
\[
 T^2{}_{12}=\tfrac12v',\qquad
 \nabla_{\bar1}T^2{}_{12}=\tfrac14v''=\tfrac14q,
\]
so reduced terminal holonomy does not force the relevant torsion
derivative to vanish.

Kotschwar's argument for Ricci flow closes a homogeneous PDE--ODE
system for projected curvature and transported projection errors
\cite{Khol}. Enlarging the present errors by torsion quantities does
not automatically repair the argument: the added errors need not
have zero terminal data. The compact metric uniqueness argument
below does not supply the missing homogeneous system for these errors.

\subsection{Gauduchon connections}\label{sec:gaud}
We keep the first Chern--Ricci flow and change the connection
whose holonomy is measured. The following construction shows that
no fixed non-Chern member of the Gauduchon family has a universal
holonomy-preservation property for this flow.

Correa--Galindo--Grama constructed examples where Gauduchon
Ricci-flatness changes under Chern--Ricci flow \cite{CGG}.
In their examples the entire Ricci form is of type \((1,1)\);
their Theorem A and formulas (40)--(42) therefore imply changes
of restricted determinant holonomy for \(\kappa<1\).
Vanishing trace curvature alone does not identify the full
holonomy group or remove its discrete determinant monodromy.
Theorem~\ref{thm:gaud} supplies exact full groups on a complex
two-torus for every \(\kappa\ne1\), including \(\kappa>1\), with
immortal uniform control and a prescribed interior time.
\subsubsection{The scalar flow and its endpoint map}
Work on \(M=\mathbb C^2/(2\pi\mathbb Z+2\pi i\mathbb Z)^2\), and let
\(x=\operatorname{Re}z^1\). For a real smooth periodic function \(f(x)\), set
\[
 g(t)=\operatorname{diag}(e^{u(t,x)},e^{f(x)}).
\]
Every such metric is globally conformally K\"ahler:
\(e^{-f}\omega(t)\) is closed, since its first coefficient depends
only on \(x=\Re z^1\) and its second coefficient is constant.
Its Lee form is the fixed exact form \(df\).
The first Chern--Ricci flow is precisely
\[
 u_t=\tfrac14e^{-u}(u+f)_{xx}.
\]
Writing \(h=u+f\), this becomes
\begin{equation}\label{eq:crh-3-1}
 h_t=\tfrac14e^{f-h}h_{xx},\qquad h(0)=v+f.
\end{equation}
The maximum principle bounds \(h\) between the extrema of \(v+f\).
Consequently \eqref{eq:crh-3-1} is uniformly parabolic while the solution exists.
Interior H\"older estimates on sliding time intervals first control
\(h\), and hence the coefficient \(e^{f-h}/4\), uniformly.
Schauder estimates and bootstrapping then control all spatial
derivatives, with constants independent of the final time
\cite{Lieberman}. Continuation gives a smooth immortal solution;
smooth initial data supply the estimates near \(t=0\).
The equation then also controls each fixed order of time derivatives.
In particular all metric curvature and torsion derivatives are uniformly
bounded. The conserved mass and a strict Lyapunov functional are
\begin{equation}\label{eq:crh-3-2}
 \mathcal M=\int e^{h-f}\,dx,\qquad
 \mathcal E(h)=\tfrac12\int h_x^2\,dx,\qquad
 \frac{d}{dt}\mathcal E(h)=-\tfrac14\int e^{f-h}h_{xx}^2\,dx.
\end{equation}
Poincar\'e's inequality and uniform parabolicity give exponential decay of
\(\mathcal E\). Uniform higher derivative estimates and the mass identity
then imply smooth convergence
\begin{equation}\label{eq:crh-3-3}
 h(t)\longrightarrow c_\infty,\qquad
 e^{c_\infty}=\frac{\mathcal M}{\int e^{-f}\,dx}.
\end{equation}

Fix \(\tau>0\) and \(\kappa\ne1\). Choose a positive integer \(N\) as follows:
if \(0<\kappa<1\), require
\begin{equation}\label{eq:crh-3-4}
 \rho=e^{-N^2\tau/4}<1-\kappa;
\end{equation}
otherwise any \(N\ge1\) will do.
Let \(X\) be the real Banach space of even, zero-mean, \(2\pi/N\)-periodic
\(C^{4,\alpha}\) functions, where \(0<\alpha<1\), and let \(\Pi\) subtract
the mean. Near \((v,f)=(0,0)\), define
\begin{equation}\label{eq:crh-3-5}
 \mathcal F(v,f)=\Pi\bigl(h_{v,f}(\tau)+(\kappa-1)f\bigr)\in X.
\end{equation}

\begin{lemma}\label{lem:endpoint}
The map in \eqref{eq:crh-3-5} is continuously differentiable near zero, and
\begin{equation}\label{eq:crh-3-6}
 D_f\mathcal F(0,0)=(\kappa-1)\Id+e^{\tau\partial_x^2/4},\qquad
 D_v\mathcal F(0,0)=e^{\tau\partial_x^2/4}.
\end{equation}
Its first derivative in \(f\) is a bounded isomorphism of \(X\).
There is a branch \(f=f(v)\), with \(f(0)=0\), satisfying
\begin{equation}\label{eq:crh-3-7}
 h_{v,f(v)}(\tau)+(\kappa-1)f(v)=c(v)
\end{equation}
for a spatial constant \(c(v)\).
For \(v=\varepsilon\cos(Nx)\),
\begin{equation}\label{eq:crh-3-8}
 f_\varepsilon=
 -\frac{\rho}{\kappa-1+\rho}\,\varepsilon\cos(Nx)
 +O_{C^{4,\alpha}}(\varepsilon^2).
\end{equation}
Every \(f_\varepsilon\) in this branch is smooth.
\end{lemma}

\begin{proof}
The local solution map for \eqref{eq:crh-3-1} follows, including differentiable dependence
on \((v,f)\), by applying the Banach implicit function theorem to
\[
 (h,v,f)\longmapsto
 \left(h_t-\tfrac14e^{f-h}h_{xx},\ h(0)-v-f\right)
\]
from \(C^{4+\alpha,2+\alpha/2}\) in the solution variable to
\(C^{2+\alpha,1+\alpha/2}\times C^{4,\alpha}\), with parameters
\(v,f\in C^{4,\alpha}\); all spaces are periodic in \(x\).
Multiplication and the exponential are smooth maps in these spaces.
There are no spatial boundary compatibility conditions on the circle.
At the zero solution the derivative
in \(h\) is the heat operator together with its initial trace, a bounded
isomorphism by the parabolic Schauder theorem. For parameters in a small
neighborhood, the solution exists on the fixed interval \([0,\tau]\).
Evaluation at \(\tau\) is bounded into \(C^{4,\alpha}\).
Linearizing \eqref{eq:crh-3-1} at zero gives the heat equation, with initial value
\(\dot v+\dot f\); the coefficient variation multiplies \(h_{xx}=0\).
This proves \eqref{eq:crh-3-6}.

The heat semigroup \(e^{\tau\partial_x^2/4}\) is compact on \(X\).
Its eigenvalues are
\(e^{-j^2N^2\tau/4}\), \(j=1,2,\ldots\), on the cosine modes.
For \(\kappa\le0\) or \(\kappa>1\), none of the numbers
\(\kappa-1+e^{-j^2N^2\tau/4}\) vanishes. For \(0<\kappa<1\), condition
\eqref{eq:crh-3-4} gives the same conclusion. The Fredholm alternative proves
invertibility on \(X\). The implicit function theorem now gives \eqref{eq:crh-3-7},
and differentiating the branch yields \eqref{eq:crh-3-8}. Smoothness of the nonlinear
maps also gives the quadratic remainder in \eqref{eq:crh-3-8}.

For completeness, regularity of the branch improves beyond the initial
H\"older space. If \(f,v\in C^{k,\alpha}\), for an even integer
\(k\ge4\), then \(h\in C^{k+\alpha,(k+\alpha)/2}\) on the initial
compact slab. Regard \(A=\tfrac14e^{f-h}\) as a given coefficient
of the linear equation \(h_t-Ah_{xx}=0\). It has the same parabolic
H\"older regularity. Interior Schauder estimates, on a smaller slab
ending at \(\tau\), give
\(h\in C^{k+2+\alpha,(k+2+\alpha)/2}\), and hence
\(h(\tau)\in C^{k+2,\alpha}\).
There is no inhomogeneous term containing derivatives of \(f\).
Equation \eqref{eq:crh-3-7}, with \(\kappa-1\ne0\), therefore yields
\(f\in C^{k+2,\alpha}\). Starting from \(k=4\) and iterating proves
\(f\in C^\infty\); the chosen \(v\) is already smooth.
\end{proof}

\subsubsection{The holonomy calculation}
Use the convention
\[
 \nabla^\kappa=\tfrac{1+\kappa}{2}\nabla^{\mathrm{Ch}}
              +\tfrac{1-\kappa}{2}\nabla^{\mathrm B};
 \qquad \kappa=1\text{ is Chern},\quad \kappa=-1\text{ is Bismut}.
\]
Set \(y=\operatorname{Im}z^1\), \(z^2=x^2+iy^2\),
\[
 J_2=E_{21}-E_{12},\qquad J_3=i(E_{21}+E_{12}),\qquad
 k=\tfrac{1-\kappa}{4}f'e^{(f-u)/2}.
\]
In the global unitary frame
\((e^{-u/2}\partial_{z^1},e^{-f/2}\partial_{z^2})\), direct calculation gives
\begin{equation}\label{eq:crh-3-9}
 \mathcal A^\kappa=
 \tfrac i2\operatorname{diag}(u',\kappa f')\,dy+
 kJ_2\,dx^2+kJ_3\,dy^2.
\end{equation}
In particular,
\begin{equation}\label{eq:crh-3-10}
 \mathcal F^\kappa_{x,x^2}=k'J_2,\quad
 \mathcal F^\kappa_{x,y^2}=k'J_3,\quad
 \operatorname{tr}\mathcal F^\kappa
  =\tfrac i2(u''+\kappa f'')\,dx\wedge dy.
\end{equation}

\begin{theorem}\label{thm:gaud}
For every \(\kappa\in\mathbb R\setminus\{1\}\) and every \(\tau>0\),
there is a smooth immortal first Chern--Ricci flow of globally
conformally K\"ahler metrics, with fixed Lee form and uniform bounds
of every order, such that
\begin{equation}\label{eq:crh-3-11}
 \Hol(\nabla^\kappa(t))=\Hol^0(\nabla^\kappa(t))=
 \begin{cases}
 SU(2),&t=\tau,\\
 U(2),&t\ge0,\ t\ne\tau .
 \end{cases}
\end{equation}
The limiting metric has full and restricted \(\nabla^\kappa\)-holonomy
\(U(2)\).
\end{theorem}

\begin{proof}
Take \(v=\varepsilon\cos(Nx)\) and \(f=f_\varepsilon\) from Lemma
\ref{lem:endpoint}, for a sufficiently small nonzero \(\varepsilon\).
Equation \eqref{eq:crh-3-8} gives \(f''(0)\ne0\), while evenness gives \(f'(0)=0\).
At this fixed point and at every time,
\[
 k'(0)=\tfrac{1-\kappa}{4}e^{(f(0)-u(t,0))/2}f''(0)\ne0.
\]
The first two curvature values in \eqref{eq:crh-3-10}, and their bracket, generate
\(\su(2)\). Thus restricted holonomy contains \(SU(2)\) at every time.

Equation \eqref{eq:crh-3-7} is equivalent to \(u(\tau)+\kappa f=c\).
The entire globally defined form \eqref{eq:crh-3-9} is therefore traceless at \(\tau\).
This proves equality with \(SU(2)\) for full and restricted holonomy there.

At any time when \(u+\kappa f\) is nonconstant, periodicity and the last
identity of \eqref{eq:crh-3-10} give nonzero trace curvature somewhere. Together with
\(\su(2)\), it generates \(\mathfrak u(2)\), hence full \(U(2)\).
It remains to exclude a second time \(t\ne\tau\) with
\(u(t)+\kappa f=c'\). Conservation of \(\int e^u\) forces \(c'=c\),
so \(h(t)=h(\tau)\) as functions. But the endpoints would then have the
same energy in \eqref{eq:crh-3-2}. Its dissipation forces \(h_{xx}=0\) throughout the
interval between them, contradicting the nonconstant endpoint
\(h(\tau)=c+(1-\kappa)f\). Thus the exceptional time is unique.

Finally, \eqref{eq:crh-3-3} gives \(u_\infty+\kappa f=c_\infty+(\kappa-1)f\), a nonconstant
function. The same curvature argument proves limiting holonomy \(U(2)\).
\end{proof}

The trace-curvature zero is transverse. At \(x=0,t=\tau\),
\begin{equation}\label{eq:crh-3-12}
 \partial_t(u''+\kappa f'')
 =\tfrac{1-\kappa}{4}e^{-c}(e^{\kappa f}f'')''
 =-\frac{(1-\kappa)\rho N^4}{4(\kappa-1+\rho)}
    \varepsilon+O(\varepsilon^2)\ne0.
\end{equation}
This is additional local information; the energy argument proves the
global uniqueness of the exceptional time.

\subsubsection*{The Bismut case: an isolated HKT time}
For \(\kappa=-1\), the terminal relation is \(u(\tau)=f+c\), so
\[
 g(\tau)=e^f\operatorname{diag}(e^c,1).
\]
This metric is globally conformal to a constant hyperk\"ahler metric on
the four-torus, so it is hyperhermitian for the same constant
hypercomplex triple. Every hyperhermitian metric in real dimension
four is HKT; see \cite[Section 3]{GP}.
This is also visible directly in \eqref{eq:crh-3-9}: at \(t=\tau\),
the Bismut connection form in the displayed unitary frame is
\(\mathfrak{su}(2)=\mathfrak{sp}(1)\)-valued and preserves that triple.
At every other finite time, its Bismut holonomy is \(U(2)\), so it
cannot preserve an orthogonal complex structure anticommuting with
the fixed complex structure. This gives an isolated HKT time along
the Chern--Ricci flow. The torsion is not closed in this diagonal
example, since \(f\) is nonconstant.
\subsection{The torsion-subtracted connection}
Let \(D^\lambda_XY=\nabla^{\mathrm{Ch}}_XY-\lambda T(X,Y)\) for
\(X,Y\in T^{1,0}M\), with \((D^\lambda)^{0,1}=\bar\partial\).
These connections are generally not metric.
In the holomorphic coordinate frame their forms are
\begin{equation}\label{eq:crh-3-13}
 A^\lambda=
 \tfrac12\operatorname{diag}(u',(1-\lambda)f')\,dz^1+
 \tfrac{\lambda f'}2E_{21}\,dz^2.
\end{equation}
Their determinant forms differ from the determinant form of
\(\nabla^{1-\lambda}\) by the globally exact real form
\begin{equation}\label{eq:crh-3-14}
 \operatorname{tr}A^\lambda-\operatorname{tr}\mathcal A^{1-\lambda}
 =\tfrac12d\bigl(u+(1-\lambda)f\bigr).
\end{equation}
Thus Theorem~\ref{thm:gaud}, with \(\kappa=1-\lambda\), also gives a
determinant holonomy contraction from \(U(1)\) to the identity and back
for every \(\lambda\ne0\). This is a corollary of \eqref{eq:crh-3-11}, not an independent
construction.

For \(\lambda=0\), the connection is the Chern connection,
and Section~\ref{sec:crossing} already gives a counterexample.
Thus no fixed member of either family universally preserves
holonomy under the first Chern--Ricci flow.

\subsection{Uniqueness and sufficient conditions for preservation}
The finite-time determinant-flatness assertion at the end of
Section~\ref{sec:crossing} uses compact backward uniqueness for
the first Chern--Ricci flow. This has a scalar formulation.
Suppose \(g(t)\) and \(\widehat g(t)\) are two smooth solutions
on \([0,\tau]\) with \(g(\tau)=\widehat g(\tau)\), and define
\[
 \phi(t)=\int_\tau^t\log\frac{\det g(s)}{\det\widehat g(s)}\,ds.
\]
Their equations give
\(g(t)-\widehat g(t)=\partial\bar\partial\phi(t)\), and hence
\[
 \phi_t=a^{i\bar j}(t,x)\partial_i\partial_{\bar j}\phi,
 \qquad \phi(\tau)=0,\qquad
 a^{i\bar j}=\int_0^1
       \bigl((1-r)\widehat g+rg\bigr)^{i\bar j}\,dr.
\]
On the compact time slab these coefficients are smooth and
uniformly elliptic. Backward uniqueness for this scalar linear
equation, for example \cite[Theorem 3]{Kback} with the ODE variable
absent, implies \(\phi=0\), and therefore \(g=\widehat g\).
In applying that theorem, choose a fixed background metric and
write the scalar operator in divergence form; the difference
is a bounded first-order term.
Forward uniqueness then extends equality beyond the meeting time.
Consequently, if the first Chern--Ricci tensor vanishes at a finite
time, comparison with the constant solution makes the flow stationary.

Another standard-tools consequence concerns parallel torsion.
Ni's theorem \cite[Theorem 1.2]{Ni} says that an irreducible restricted
Hermitian holonomy representation with parallel torsion and nonzero first
Ricci curvature forces the metric to be K\"ahler.
For the Chern connection this gives a simple preservation regime:
initial irreducible restricted holonomy and parallel Chern torsion lead either
to a stationary Chern--Ricci-flat solution or to K\"ahler--Ricci flow,
where the Ricci-flow holonomy theorem applies.
The reducible parallel-torsion case requires additional analysis;
no classification of it is asserted here.

There is also a direct sufficient condition that avoids a
classification. If
\(\nabla^{\mathrm{Ch}}_{g_0}\Ric^{(1)}(g_0)=0\), then
\[
 g(t)=g_0-t\Ric^{(1)}(g_0)
\]
solves the flow for as long as this form is positive, and its
Chern connection is exactly independent of time.
Indeed the initial Chern connection preserves every displayed
metric and has the same \((0,1)\)-part; uniqueness of the Chern
connection identifies it with the evolving connection. Its trace
curvature is therefore fixed, which verifies the equation.
Compact uniqueness identifies any flow with this explicit one.
In particular, this applies when both Chern torsion and Chern
curvature are initially parallel. This is an elementary sufficient
condition, not a general necessity criterion.

\appendix
\section{Chern holonomy in Yang's Hopf solution}\label{app:hopf}
We compute the full and restricted Chern holonomy of a previously
published first Chern--Ricci flow, to clarify the comparison in
the introduction. Let
\[
 M=(\mathbb C^2\setminus\{0\})/\langle z\mapsto\mu z\rangle,
 \qquad |\mu|>1,
\]
and put \(\rho=|z|^2\), \(P_z=zz^*/\rho\), and
\(g_0=\rho^{-1}I\). All matrices in this appendix represent
the squared norm \(v^*gv\).

\subsection{The Chern holonomy of the metric family}
A spatially constant rescaling does not change the Chern connection.
Consider
\[
 h_c=\rho^{-1}(I+cP_z),\qquad c>-1.
\]
Direct multiplication gives
\begin{equation}\label{eq:crh-B-4}
 \Gamma_k=h_c^{-1}\partial_kh_c
 =-\frac{\bar z_k}{\rho}I+
   \frac c\rho e_kz^*-\frac{c\bar z_k}{\rho}P_z.
\end{equation}
At the point \(z=(1,0)\), the only nonzero complex curvature component is
\begin{equation}\label{eq:crh-B-5}
 R_{2\bar2}=\diag(1+c,1-c).
\end{equation}
More invariantly,
\[
 R_{k\bar l}
 =\left(\frac{\delta_{kl}}{\rho}
          -\frac{\bar z_k z_l}{\rho^2}\right)
    \bigl((1-c)I+2cP_z\bigr).
\]
At \(z=(1,0)\), differentiation after parallel transport gives
\begin{equation}\label{eq:crh-B-6}
 D_{\bar2}R_{2\bar2}=2cE_{12},\qquad
 D_2R_{2\bar2}=2c(1+c)E_{21}.
\end{equation}
For \(c\ne0\), these lie in the complexified restricted holonomy algebra
and generate \(\mathfrak{sl}(2,\mathbb C)\). Equation \eqref{eq:crh-B-5} has nonzero
trace, so the complexification is \(\mathfrak{gl}(2,\mathbb C)\).
Since the Chern connection is unitary, its real holonomy algebra is
\(\mathfrak u(2)\), and full and restricted holonomy are \(U(2)\).

At \(c=0\), the connection form is the scalar
\(-\partial\log\rho\,I\), with nonzero scalar curvature.
Both the lifted parallel transport and the deck derivative \(\mu I\)
are scalar. Unitarity then gives full and restricted holonomy equal
to the central \(U(1)\).

\subsection{An interior Chern crossing in Yang's explicit solution}
The solution originating in \cite{TW}, in the parametrization
of \cite[formula (2.2)]{YangHopf}, is in complex dimension two
\begin{equation}\label{eq:crh-B-7}
 g_Y(t)=\rho^{-1}\bigl[(1+T_0-2t)I-(T_0-2t)P_z\bigr],
 \qquad 0\le t<(T_0+1)/2,
\end{equation}
where \(T_0>0\). Its determinant is
\((1+T_0-2t)\rho^{-2}\), so direct differentiation verifies the
first Chern--Ricci equation. Up to the spatially constant positive
factor \(1+T_0-2t\), it equals \(h_{c(t)}\), where
\begin{equation}\label{eq:crh-B-8}
 c(t)=\frac{2t-T_0}{1+T_0-2t}>-1.
\end{equation}
Thus the preceding calculation gives the full and restricted groups
\[
 \Hol(g_Y(t))=\Hol^0(g_Y(t))=
 \begin{cases}
 U(2),&t\ne T_0/2,\\
 U(1)\text{ (central)},&t=T_0/2.
 \end{cases}
\]
This is an interior \(U(2)\)-to-\(U(1)\)-to-\(U(2)\) crossing
in a previously published explicit family. Its Chern-parallel
endomorphism space has dimension \(1,4,1\), respectively.

The standard Hopf metric in complex dimension two is pluriclosed:
\(\partial\bar\partial(\rho^{-1}\omega_{\mathrm{Eucl}})=0\)
follows from the harmonicity of \(\rho^{-1}\) on
\(\mathbb R^4\setminus\{0\}\).
Moreover \(g_Y(t)-g_0\) is a constant multiple of
\(\partial\bar\partial\log\rho\); hence the whole family is
pluriclosed. It has bounded geometry on each compact time interval
strictly before \((T_0+1)/2\), including intervals across the
holonomy crossing. Consequently it already rules out a Chern
analogue of finite-interval Ricci-flow holonomy preservation under
bounded curvature and torsion, even within the pluriclosed class.
The immortal and uniform all-time properties in Section~\ref{sec:crossing} are
additional features rather than the first such obstruction.

\end{document}